\documentclass[12pt,reqno]{amsart}
\usepackage[margin = 1.3 in]{geometry}
\usepackage{
  hyperref,
  amsmath,
  amssymb,
  amsthm,
  thmtools,
  microtype,
  mathrsfs,
  enumitem,
  stmaryrd,
  diagbox,
  comment
}
\usepackage[table]{xcolor}
\usepackage{tikz, tikz-cd}

\usepackage{graphicx}

\usepackage{eucal}

\usepackage{cleveref}

\theoremstyle{plain}
\newtheorem{theorem}{Theorem}[section]
\newtheorem{proposition}[theorem]{Proposition}
\newtheorem{lemma}[theorem]{Lemma}
\newtheorem{conjecture}[theorem]{Conjecture} 
\newtheorem{corollary}[theorem]{Corollary}
\newtheorem{definition}[theorem]{Definition}

\newtheorem{example}[theorem]{Example}

\newtheorem{question}[theorem]{Question}
\theoremstyle{remark}

\numberwithin{equation}{section}

\title{Lines highly tangent to a hypersurface}

\renewcommand{\k}{k}

\DeclareMathOperator{\II}{II}
\DeclareMathOperator{\Tot}{Tot}
\DeclareMathOperator{\oTot}{\overline{\Tot}}

\DeclareMathOperator{\dm}{dim}

\renewcommand{\to}{{\longrightarrow}}

\ifcsname theorem\endcsname{}\else\declaretheorem[parent=section]{theorem}\fi
\ifcsname corollary\endcsname{}\else\declaretheorem[sibling=theorem]{corollary}\fi
\ifcsname lemma\endcsname{}\else\declaretheorem[sibling=theorem]{lemma}\fi
\ifcsname proposition\endcsname{}\else\declaretheorem[sibling=theorem]{proposition}\fi
\ifcsname conjecture\endcsname{}\else\fi
\ifcsname problem\endcsname{}\else\fi
\ifcsname question\endcsname{}\else\fi
\ifcsname definition\endcsname{}\else\declaretheorem[sibling=theorem, style=definition]{definition}\fi
\ifcsname exercise\endcsname{}\else\fi
\ifcsname example\endcsname{}\else\fi
\ifcsname remark\endcsname{}\declaretheorem[sibling=theorem, style=remark]{remark}\fi

\renewcommand {\P}{{\bf P}}

\providecommand{\codim}{\operatorname{codim}}

\providecommand{\rk}{\operatorname{rk}}
\declaretheorem[sibling=theorem,style=plain]{remark}
\numberwithin{equation}{section}

\renewcommand{\O}{\mathcal O}
\newcommand{\G}{\mathbf G}
\newcommand{\T}{\mathbf T}
\newcommand{\I}{\mathcal{I}}
\renewcommand{\k}{\mathbb K}

\newcommand{\smvee}{\raise0.5ex\hbox{$\scriptscriptstyle\vee$}}

\newcommand{\PP}{\mathbb P}

\newcommand{\Spec}{{\text{\rm Spec}\,}}

\author{Anand Patel}
\author{Eric Riedl}
\author{Geoffrey Smith}
\author{Dennis Tseng}

\begin{document}

\begin{abstract}We study spaces of lines that meet a smooth hypersurface $X$ in $\P^n$ to high order. As an application, we give a polynomial upper bound on the number of planes contained in a smooth degree $d$ hypersurface in $\P^5$ and provide a proof of a result of Landsberg without using moving frames.
\end{abstract}
\maketitle

\section{Introduction}

\label{sec:introduction}

We work in characteristic 0. Our intention in this paper is first to illustrate the simplifying
role played by log tangent sheaves in the analysis of lines highly
tangent to a hypersurface $X \subset \P^n$, and then to obtain an
upper bound on the number of $2$-planes contained in a smooth
hypersurface in $\P^5$.  The simple observation we exploit is: {\sl If
  $\ell \subset \P^n$ is a line, then sections of the restricted
  log-tangent sheaf $T_{\P^n}(- \log X)|_{\ell}$ parametrize those
  infinitesimal deformations of $\ell$ preserving the moduli of the
  point-configuration $\left(\ell,\ell \cap X \right)$.}

This observation helps simplify arguments.  To illustrate this, we
provide a quick proof of the following theorem of Landsberg without
using moving frames:

\begin{theorem}
  \label{theorem:Landsberg}[\cite{Landsberg} Theorem 2] Let $X \subset \P^{n}$ be a reduced
  hypersurface over a characteristic zero field, $p \in X$ a general
  point, and let $\Sigma_{k} \subset \G(1,n)$ denote any irreducible
  component of the variety of lines meeting $X$ to order at least $k$
  at $p$.  If $\dm \Sigma_{k}$ exceeds its expected dimension $n-k$,
  then all lines parametrized by $\Sigma_{k}$ are entirely contained
  in $X$.
\end{theorem}

A special instance of \autoref{theorem:Landsberg} is the fact that the
{\sl flec-nodal} locus $F$ of a smooth, degree $d \geq 3$ surface $S$
in $\P^{3}$ is always one dimensional, a fact initially proved by Salmon \cite{salmon}.  This fact, coupled with the
calculation of the degree of $F$, was used by Segre \cite{Segre} to show that the number of lines contained in $S$ is at most $(d-2)(11d-6)$.

Using Theorem \ref{theorem:Landsberg} together with a careful geometric analysis, we can show that for any smooth hypersurface $X$ of degree $d$ in $\PP^5$, any component of the space of lines meeting $X$ to order 5 has dimension at most 2 (see Theorem \ref{theorem:generalV5}), which generalizes the $n=5$ case of the de Jong-Debarre Conjecture \cite{Beheshti}. Related work on lines highly tangent to hypersurfaces can also be found in \cite{lloydSmith}. Using an intersection theory calculation, we obtain an explicit upper bound on the space of $2$-planes in $X$, which had for $d \geq 3$ been known to be finite \cite[Corollary in Appendix]{Browning2016}.

\begin{theorem}
Let $X \subset \PP^5$ be a smooth hypersurface of degree $d$. Then $X$ contains at most $120 d^2 - 150 d^3 + 35 d^4$ 2-planes.
\end{theorem}

The Fermat hypersurface of degree $d$ in $\PP^5$ is known to contain $15d^3$ 2-planes, which can be compared to the $O(d^4)$ bound that we obtain. See \cite{DegItenOtt} for some recent results about the space of planes on cubic 4-folds using lattice theory. In characteristic $p$, many of these results do not hold. See \cite{Cheng} for some examples of the types of behavior that can occur.

\textbf{Acknowledgements:} We gratefully acknowledge conversations with Izzet Coskun, Changho Han, Mitchell Lee, and Gwyneth Moreland. Eric Riedl was supported by NSF CAREER grant DMS-1945944.

\section{Deformations of lines with fixed contact order}
\label{sec:proof-Landsberg}

The goal of this section is to prove \autoref{theorem:Landsberg} using an analysis of the deformation theory of the space of lines meeting $X$ to high contact order. We begin by reviewing the deformation theory of these lines and its relation to the log tangent sheaf. Let $X$ be defined by the homogeneous polynomial
$F(x_0, \ldots, x_n)$, and suppose that $X$ contains the point
$p=\left[1:0 \ldots: 0 \right]$.  We let
\[T_{\P^{n}}(- \log X) \subset T_{\P^{n}}\] denote the logarithmic
tangent sheaf, i.e. the sheaf of derivations
$\O_{\P^{n}} \to \O_{\P^{n}}$ mapping the ideal sheaf $\I_{X}$ into
itself.  In general, the logarithmic tangent sheaf is reflexive if $X$
is reduced, and when $X$ is smooth (or more generally, simple normal crossing) the logarithmic tangent sheaf is a
vector bundle.  We fix homogeneous coordinates $[s:t]$ on
$\P^{1}$.

Let $\iota: \P^{1} \to \P^{n}$ be a linear embedding of $\P^{1}$.
\begin{definition}
  \label{def:deformation}
  A {\sl first order deformation} of $\iota$ is a map
\begin{align}
\iota_{\epsilon}: \P^{1} \times \Spec \k[\epsilon]/(\epsilon^{2}) \to
  \P^{n}
\end{align}
which restricts to $\iota$ modulo $(\epsilon)$. We say a first order
deformation $\iota_{\epsilon}$ {\sl preserves $\iota^{-1}(X)$} if the
scheme $\iota_{\epsilon}^{-1}(X)$ is the preimage of $\iota^{-1}(X)$
under the projection to the first factor
\[\P^{1} \times \Spec \k[\epsilon]/(\epsilon^{2}) \to \P^{1}.\]
\end{definition}

The significance of logarithmic tangent sheaves in the present setting
is captured by the following lemma.

\begin{lemma}[{\cite[Lemma 2.1]{PRT20}}]
  \label{lemma:preserveset} Global sections of
  $\iota^{*}T_{\P^{n}}(- \log X)$ are in one-to-one correspondence
  with first order deformations of $\iota$ preserving $\iota^{-1}(X)$.
\end{lemma}

We take a moment to concretely describe the correspondence in
\autoref{lemma:preserveset} using coordinates.  Suppose $[s:t]$ are
homogeneous coordinates on $\P^1$ and that the map $\iota$ is defined
by $[a_0: \hdots : a_n]$, where each $a_i$ is a linear form in $s,t$.
Then, by the Euler exact sequence, the global sections of
$\iota^{*}T_{\P^{n}}$ are identified with tuples $(b_0, \dots, b_n)$
of linear forms in $s,t$, considered modulo the tuple
$(a_0, \dots, a_n)$.

A global section of $\iota^{*}T_{\P^{n}}$ represented by such a tuple
$(b_0, \dots, b_n)$ corresponds to the first-order deformation
$\iota_{\epsilon}$ defined by
$[s:t] \mapsto [a_0 + \epsilon b_0 : \hdots : a_n + \epsilon b_n]$.

Within the space of all deformations
$H^{0}(\P^{1},\iota^{*}T_{\P^{n}})$, the subspace
$H^{0}(\P^{1},\iota^{*}T_{\P^{n}}(- \log X))$ is represented by tuples
$(b_0, \dots, b_n)$ obeying the condition
\begin{align}
  \label{eq:eulersyz}
  \sum \limits_{i=0}^{n} b_i \cdot \frac{\partial F}{\partial x_{i}}(a_0, \dots, a_{n}) \equiv 0 \mod F(a_0, \dots, a_n).
\end{align}

The next construction will effectively allow us
to replace the degree $d$ hypersurface $X$ with a degree $k<d$
hypersurface. Let $y_j = x_j/x_0$, $j=1, \ldots, n$ be affine
coordinates on the corresponding affine chart $x_0 \neq 0$ around $p$,
and let $f(y_1, \ldots, y_n)$ denote the dehomogenization of $F$ on
this chart.  We expand
$$f= f_1 + \ldots + f_d,$$ where $f_i$ is the degree $i$ part of $f$,
and we let $F_k$ be the homogenization of $f_1 + \ldots + f_k$.
Explicitly,
\begin{align}
  F_k(x_0, \ldots, x_n) = \sum_{j=1}^{k}x_{0}^{k-j}\cdot f_{j}(x_1, \ldots, x_n).
\end{align}
We let $X_k$ denote the hypersurface defined by
$F_k(x_0, \ldots, x_n)=0$, and we note that $p \in X_{k}$ is a smooth
point.  Hence, $X_k$ is generically smooth.

\begin{definition}
  \label{definition:preservecontact}
  Suppose $\iota: \P^{1} \to \P^{n}$ is a linear map such that
  $\iota^{*}F \in \k[s,t]$ is non-zero and divisible by $s^{k}$. A
  first order deformation $\iota_{\epsilon}$ of $\iota$
  \emph{preserves order $k$ contact at $[0:1]$} if
  $\iota_{\epsilon}^{*}F$ is divisible by $s^{k}$ in the ring
  $\k[\epsilon,s,t]/(\epsilon^{2})$.
\end{definition}

\begin{corollary}
  \label{corollary:ktangent} Let $\iota: \P^{1} \to \P^{n}$
  parametrize a line such that $\iota^{*}F$ is non-zero and divisible
  by $s^{k}$, and such that $\iota([0:1])=p$. Then the global sections
  of $\iota^{*}T_{\P^{n}}(- \log X_k)$ are in one-to-one
  correspondence with first order deformations $\iota_{\epsilon}$
  preserving order $k$ contact at $[0:1]$.
\end{corollary}

\begin{proof}
  Let $[s:t] \mapsto [a_0: \hdots : a_n]$ define $\iota$, and choose coordinates on $\PP^n$ so that
  $a_0$ is not a scalar multiple of $s$, but $a_{i}$, $i \geq 1$ are
  multiples of $s$. By assumption, $F(a_0, \dots, a_n)$ is a non-zero
  form divisible by $s^k$. This, in turn, implies that
  $F_k(a_0, \dots, a_n)$ is a non-zero multiple of $s^k$.

  Suppose now that $(b_0, \dots, b_n)$ is a tuple of linear forms
  representing a deformation $\iota_\epsilon$ preserving order $k$
  contact at $[0:1]$. This condition is equivalent to the condition
  that the form
  \begin{align}
    \label{eq:bF}
    \sum \limits_{i=0}^{n}b_{i} \cdot \frac{\partial F}{\partial x_i}(a_0, \dots, a_n)
  \end{align}
  is also divisible by $s^{k}$.  Next, a direct computation shows that
  \begin{align}
    \label{eq:congruence}
    \frac{\partial F}{\partial x_i}(a_0, \dots, a_n) \equiv a_0^{d-k} \frac{\partial F_{k}}{\partial x_{i}} \mod s^k
  \end{align}

  Therefore, since $a_0$ is not divisible by $s$, \eqref{eq:bF} is
  divisible by $s^k$ if and only if 
  \begin{align}
    \label{eq:bFk}
    \sum \limits_{i=0}^{n}b_{i}\cdot \frac{\partial F_{k}}{\partial x_{i}}(a_0, \dots, a_n)
  \end{align}
  is also divisible by $s^{k}$. In other words
  \begin{align*}
    \sum \limits_{i=0}^{n}b_{i}\cdot \frac{\partial F_{k}}{\partial x_{i}}(a_0, \dots, a_n) \equiv 0 \mod F_k(a_0, \dots, a_n)
  \end{align*}
  meaning that $(b_0, \dots, b_n)$ represents a global section of
  $\iota^{*}T_{\P^{n}}(- \log X_k)$, which is what we needed to prove.
\end{proof}

\subsection{Some incidence varieties}
\label{sec:dimensions}

Before recovering the result of Landsberg, we need to study the locus swept out by a family of high-contact lines. The main goal is Theorem \ref{theorem:expecteddim}.

\begin{definition}
  \label{definition:inc}
  Define the {\sl incidence variety}
  \[\mathsf{Inc} \subset \G(1,n) \times \P^{n}\] to be the set of pairs
  $(\ell,p)$ such that $p \in \ell$.

  Let
  $\lambda: \mathsf{Inc} \to \G(1,n)$ and $\pi: \mathsf{Inc} \to \P^n$
  denote the two projections.
\end{definition}

The projection $\lambda$ expresses $\mathsf{Inc}$ as the universal
line over the Grassmannian $\G(1,n)$.

\begin{definition}
  \label{definition:TVS}
  Suppose $V \subset \mathsf{Inc}$ is any subset. Define
  \[\Tot(V) \subset \G(1,n) \times \P^{n} \times \P^{n}\] to be the
  set of triples $(\ell,x,y)$ such that $(\ell,x) \in V$ and
  $y \in \ell$.  Define \[\overline{\Tot}(V) \subset \P^{n}\] to be the
  image of $\Tot(V)$ under the projection to the third factor, i.e.
  $\overline{\Tot(V)}$ is the union of all lines in $\lambda(V)$.
\end{definition}

Observe that for any subset $V \subset \mathsf{Inc}$,
\[\pi(V) \subset \overline{\Tot}(V).\]

\begin{definition}
  \label{definition:Vk}
  Let $X \subset \P^{n}$ be a degree $d$ hypersurface. For each
  $k \leq d$ define the variety
\begin{align}
  \label{eq:Vk}
  V_k(X) \subset \mathsf{Inc}
\end{align}
to be the locus of pairs $(\ell, x)$ with $x \in X$, such that $\ell$
meets $X$ at the point $x$ to order $\geq k$.
\end{definition}

If the hypersurface $X$ is understood in context, we will sometimes
write $V_k$ for $V_{k}(X)$.

\begin{theorem}
  \label{theorem:expecteddim} Let $X \subset \P^n$ be a smooth
  hypersurface. If $W \subset V_k$ is an irreducible component
  satisfying \[ \overline{\Tot}(W) = \P^{n},\] then $W$ has the
  expected dimension $2n-k-1$.
\end{theorem}

\begin{proof}
  That every irreducible component of $V_k$ has dimension at least
  $2n-k-1$ is a simple dimension count which we omit. It is
  enough to show therefore that the tangent space $T_{W}|_{(\ell,p)}$
  of $W$ at a general point $(\ell,p) \in W$ has dimension at most
  $2n-k-1$. Further, it is enough to prove the result only for components where the general line meets $X$ at $p$ with multiplicity exactly $k$.
  
  To that end, suppose $(\ell,p)$ is a general point of $W$, and that
  \[\iota: \P^{1} \to \P^{n}\] is a parametrization of $\ell$ sending
  $[0:1]$ to $p$, as in the setting of \autoref{corollary:ktangent}.
  Notice that we may assume that $\ell$ meets $X$ at $p$ with
  multiplicity exactly equal to $k$.

  We invoke the auxiliary hypersurface $X_k$ constructed in the last
  section.  Since $\ell$ meets $X_k$ only at $p$, and since $p$ is a
  smooth point of $X_k$, we deduce that
  $\iota^{*}T_{\P^{n}}(- \log X_{k})$ is a locally free on $\P^{1}$,
  and its global sections parametrize first order deformations of the
  map $\iota$ preserving order $k$ contact at $[0:1]$ by
  \autoref{corollary:ktangent}.

  The condition \[\overline{\Tot}(W) = \P^{n}\] implies that the map
  which sends a deformation $\iota_{\epsilon}$ preserving order $k$
  contact at $[0:1]$ to the deformation induced on the point
  $\iota(y) \in \P^{n}$ for a general $y \in \P^{1}$ is surjective.
  This means that the restriction map
\begin{align}
  \label{eq:restriction}
  H^{0}(\P^{1},\iota^{*}T_{\P^{n}}(- \log X_k)) \to \iota^{*}T_{\P^{n}}(- \log X_k)|_{y} = \iota^{*}T_{\P^{n}}|_{y}
\end{align}
is surjective.  We conclude that $\iota^{*}T_{\P^{n}}(- \log X_{k})$
is a globally generated vector bundle.  Since
\[\deg \iota^{*}T_{\P^{n}}(- \log X_{k}) = n+1-k,\] and
\[\rk \iota^{*}T_{\P^{n}}(- \log X_{k}) = n,\] it follows from global
generation that
\[\dm H^{0}(\P^{1},\iota^{*}T_{\P^{n}}(- \log X_{k}))=2n-k+1.\]

Finally, we note that there is a $2$-dimensional space of deformations
of the map $\iota$ which induce a trivial deformation on the pair
$(\ell, p)$ (There is a $2$-dimensional group of automorphisms of
$\P^{1}$ fixing a point).  Since every tangent vector of $W$ at the
point $(\ell, p)$ can be lifted to a first order deformation of
$\iota$ preserving order $k$ contact at $[0:1]$, we conclude that
$\dm T_{W}|_{(\ell,p)} \leq (2n-k+1)-2$, as desired.
\end{proof}

We are now in position to provide the proof of
\autoref{theorem:Landsberg}:

\begin{proof}[Proof of \autoref{theorem:Landsberg}]
  Assume the setting in the statement of
  \autoref{theorem:Landsberg}. Suppose the dimension of $\Sigma_{k}$
  exceeds $n-k$.  Since $p \in X$ is a general point, this implies the
  existence of an irreducible component \[W \subset V_k\] containing
  $\Sigma_{k}$ having dimension exceeding $(n-k)+(n-1) = 2n-k-1$.

  \autoref{theorem:expecteddim} now implies
  $\overline{\Tot}(W) \neq \P^{n}$. Since $X = \pi(W)$ is a subvariety
  of $\overline{\Tot}(W)$, it follows that $\overline{\Tot}(W) = X$,
  which is the theorem's conclusion.
  \end{proof}

  \section{Fifth order tangent lines}
\label{sec:quintuple-contact-lines}

This section culminates in the proof of \autoref{theorem:generalV5},
which states that the variety $V_5(X)$ has the expected dimension for
any degree $d \geq n$ smooth hypersurface $X \subset \P^{n}$.  We note
here that similar statements for $V_{k}(X)$, $k \leq 4$ are also true
-- $k=1$ is trivial, $k=2$ follows because $X$ is smooth, and $k=3$
follows from \autoref{theorem:fundamentalform}.  Finally, the case
$k=4$ can be proved using a simpler version of the proof in
\autoref{theorem:generalV5}, so we omit it.

\autoref{theorem:generalV5} will then be used in
\autoref{sec:boundinghalfspaces} to provide a bound on the number of
$m$-planes that a sufficiently large degree $2m$-dimensional
hypersurface can contain.

After recalling some classical facts about the second fundamental form
and the classification of varieties with many lines, we will prove a
basic fact about cones on scrolls, and then prove the main theorem of
the section, \autoref{theorem:generalV5}.

\subsection{Preliminaries on the second fundamental form}
\label{sec:prel-facts-about}

Let $x$ be a point on a smooth hypersurface $X \subset \P^{n}$ of
degree $d \geq 2$. Letting $\mathfrak{m}_x$ denote the maximal ideal of
$x \in X$, since $X$ is smooth, the vector space
$$H^{0}(X, \mathfrak{m}_{x}^{2}\otimes \O_{X}(1))$$ is $1$-dimensional,
generated by an element $s_x$.  Its residue in the $\k$-vector space
$\left( \mathfrak{m}_{x}^{2}/\mathfrak{m}_{x}^{3} \right)(1)$ defines a
quadratic form (well-defined up to scaling, and possibly zero) denoted
\[\II_{x}\] on the Zariski tangent space
$T_{x}X = \left(\mathfrak{m}_{x}/\mathfrak{m}_{x}^{2}\right)^{\vee}$. The
quadratic form $\II_x$ is called the {\sl second fundamental form} of
$X$ at $x$.

We will need the following fundamental fact about the behavior of
$\II_x$:

\begin{theorem}(\cite{3264} Theorem 7.11)
  \label{theorem:fundamentalform}
  For each $0\leq i \leq n-1$, the closed subset \[R_{i} \subset X\]
  consisting of points $x \in X$ where \[\rk \II_{x} \leq (n-1)-i\] has
  codimension $\geq i$.
\end{theorem}

\subsection{Fourfolds with many lines}
\label{sec:fourfolds-with-many}
 
First, we record a well-known proposition about sweeping families of
lines on hypersurfaces.  For any hypersurface $X$, we let $F_{1}(X)$
denote the Fano scheme of lines contained in $X$. For any subscheme
$A \subset F_{1}(X)$, we let $\Tot(A) \to A$ denote the universal line
over $A$; $\Tot(A)$ is a $\P^{1}$-bundle over $A$. Furthermore, we
denote by $\oTot(A)$ the union of all lines parametrized by $A$.

\begin{proposition}
  \label{proposition:sweeping}
  Suppose $X \subset \P^{n}$ is a smooth hypersurface, and suppose an
  irreducible component $B \subset F_{1}(X)$ sweeps out $X$,
  i.e. $\oTot(B) = X$. The normal bundle $N_{\ell/X}$ is globally
  generated for a general line $\ell \in B$.
\end{proposition}

\begin{proof}
  Let $(\ell,p) \in \Tot(B)$ denote a general point. Then the
  derivative of the natural projection $\pi: \Tot(B) \to X$ is a map
  on Zariski tangent spaces:
  \begin{align}
    \label{eq:tangentspaces}
    \psi: T_{(\ell,p)}\Tot(B) \to T_{p}X.
  \end{align}
  The sweeping hypothesis implies that \eqref{eq:tangentspaces} is
  surjective.  Furthermore, \eqref{eq:tangentspaces} sends the
  ``vertical tangent subspace'' space
  \[T_{p}\ell \subset T_{(\ell,p)}\Tot(B)\] isomorphically to
  $T_{p}\ell \subset T_{p}X$, and therefore we get an induced
  surjection on quotient spaces:
  \begin{align}
    \label{eq:surjection}
    T_{(\ell,p)}\Tot(B)/T_p\ell \to N_{\ell/X}|_{p}.
  \end{align}
  The domain vector space in \eqref{eq:surjection} is canonically
  identified with the Zariski tangent space $T_{[\ell]}B$, which by
  standard deformation theory is canonically isomorphic to
  $H^{0}(\ell, N_{\ell/X})$.  Under these identifications,
  \eqref{eq:surjection} becomes the restriction map
  \begin{align*}
    H^{0}(\ell, N_{\ell/X}) \to N_{\ell/X}|_{p},
  \end{align*}
  and its surjectivity is precisely what we needed to prove.
\end{proof}

\begin{corollary}
  \label{cor:sweeping}
  If $X \subset \P^{n}$ is a smooth hypersurface of degree $d \geq n$
  then $X$ is not swept out by lines.
\end{corollary}

\begin{proof}
  The normal bundle of any line $\ell \subset X$ has degree
  $n-d-1 <0$, and therefore cannot be globally generated.  The
  corollary now follows from \autoref{proposition:sweeping}.
\end{proof}

We will also need the following theorem of Segre classifying four dimensional
projective varieties containing a five dimensional family of lines:

\begin{theorem}[Theorem 1 of \cite{Rogora}]
  \label{theorem:fivelines}
  If $X \subset \P^{N}$ is a variety which contains a
  $(2\dim X-3)$-dimensional family of lines, then $X$ is either a
  linear space, a quadric, or a scroll.
\end{theorem}

\subsection{Cones in scrolls}
\label{sec:linesonscrolls}

In this section, we record a basic result on cones contained in
scrolls, which we were unable to find a reference for.

\begin{definition}
  \label{definition:scroll}
  An $n$-dimensional variety $P \subset \P^{N}$ is a {\sl scroll} if
  there exists a smooth, irreducible, projective curve $B$, a
  $\P^{n-1}$-bundle $\pi: \Lambda \to B$ and a morphism
  $\varphi: \Lambda \to \P^{N}$ satisfying:
  \begin{enumerate}
  \item $\varphi$ restricts to a linear embedding on each fiber of
    $\pi$, and
  \item $\varphi(\Lambda) = P$, and $\varphi$ is birational onto its
    image.
  \end{enumerate}
\end{definition}

If $P$ is a scroll, and if $\pi: \Lambda \to B$ is understood, we call
the linear spaces $\varphi(\Lambda_{b}) \subset \P^{N}$ the {\sl
  rulings} of $P$. To ease notation, we will denote
$\varphi(\Lambda_{b})$ by just $\Lambda_{b}$ as long as no confusion
arises.

\begin{definition}
  \label{definition:cone} A {\sl cone} in $\P^N$ is a pair $(C,c)$
  where $C \subset \P^N$ is a projective variety and $c \in C$
  satisfying: For any $q \in C \setminus \{c \}$, the line joining $q$
  and $c$ is entirely contained in $C$.  If $(C,c)$ is a cone, we say
  a line $\ell \subset C$ is a {\sl cone line} if $c \in \ell$.
\end{definition}

Note that if $(C,c)$ is a cone, and if $H \subset \P^N$ is any
hyperplane containing $c$, then $(C \cap H, c)$ is again a cone.
Furtheremore, if a cone $(C,c)$ is irreducible and if $\dim C \geq 3$,
then for a general hyperplane $H$ containing $c$, the cone
$(C \cap H, c)$ is again irreducible.

The next proposition constrains large cones in scrolls.

\begin{proposition}
  \label{proposition:scrolls}
  Suppose $P \subset \P^{N}$ is an $n$-dimensional scroll which is not
  a linear space, with rulings $\Lambda_{b}, b \in B$. Let $p \in P$,
  and suppose $(C,p)$ is a $(n-1)$-dimensional irreducible cone
  contained in $P$.  Assume only finitely many of the rulings contain
  $p$.

  Then $C$ is a linear space.
\end{proposition}

\begin{proof}

  Assume, for sake of contradiction, that $\deg C >1$.  Then, as
  $\dim C = n-1$, it follows that $C$ is not contained in any ruling
  $\Lambda_{b}$.  Therefore, a general cone line $\ell \subset C$
  meets a general $\Lambda_{b}$ at a single point. Since this is true
  for a general $\ell$, it follows that {\sl every} $\ell$ meets a
  general $\Lambda_{b}$ -- furthermore, this meeting is a single
  point, because $p$ is only contained in finitely many rulings.

  Now pick any hyperplane $H$ avoiding $p$, and consider projection
  from $p$
  \[\pi_{p} : P \setminus \{p\} \to H.\]
  The image $W := \pi_{p}(C \setminus \{p\})$ is the $n-2$-dimensional
  variety $C \cap H$, which has degree $>1$ by
  assumption. Furthermore, the image $\pi_{p}(\Lambda_{b})$ for
  $b \in B$ general, is a $(n-1)$-dimensional linear space in $H$,
  and \[W \subset \pi_{p}(\Lambda_{b}).\] However, as $\deg W >1$,
  this implies that $\pi_{p}(\Lambda_{b})$ is independent of $b$.
  Thus, the image of $\pi_{p}:P \setminus \{p\}$ is the
  $(n-1)$-dimensional linear space $\pi_{p}(\Lambda_{b})$, forcing $P$
  itself to be a linear space, contrary to assumption.  The
  argument by contradiction is complete.
\end{proof}

We can now prove the main theorem of this section.

\begin{theorem}
  \label{theorem:generalV5} If $X \subset \P^{n}$ is a smooth
  hypersurface of degree $d \geq n \geq 4$, then $V_5(X)$ has the expected
  dimension $2n-6$.
\end{theorem}

\begin{proof}
  We proceed by contradiction, and suppose here onward that
\[W \subset V_{5}\] is an irreducible subvariety satisfying
\[\dim W = 2n-5.\]
We let $\pi: W \to X$ denote the map sending a pair $(\ell,x)$ to the
point $x \in X$.  Observe that for every $x \in X$, the preimage
\[\pi^{-1}(x) \subset W\] is naturally a subvariety of the $(n-2)$-dimensional projective space
parametrizing lines in $\P^{n}$ which are tangent to $X$ at $x$.
Thus, if we set
\[Y := \pi(W) \subset X,\] there are three possibilities:
\begin{enumerate}
\item  $Y = X$, or
\item  $Y \subset X$ has codimension $2$, or
\item $Y \subset X$ has codimension $1$.
\end{enumerate}

The proof proceeds by eliminating each of these possibilities.

\begin{lemma}
  \label{lemma:YneqX} $Y \neq X$.
\end{lemma}

\begin{proof}
  Direct application of \autoref{theorem:Landsberg}.
\end{proof}

\begin{lemma}
  \label{lemma:Ycodim2} $Y \subset X$ cannot have codimension $2$.
\end{lemma}

\begin{proof}
  By contradiction. Assume $\codim Y = 2$. In this case, for a general
  (and hence every) point $y \in Y$, the fiber $\pi^{-1}(y) \subset W$
  has dimension $(n-2)$.  Since, for every $(\ell,x) \in W$, the line
  $\ell$ meets $X$ at $x$ with contact order $>2$, it follows that the
  second fundamental form $\II_{x}(X)$ vanishes identically (has rank
  $0$) for all points $y \in Y$. This violates
  \autoref{theorem:fundamentalform}, providing the desired
  contradiction.
\end{proof}

Thus, we may and shall assume $Y \subset X$ has codimension $1$ from
here onward.  First, we determine the dimension of
$\oTot(W)$. \autoref{theorem:expecteddim} implies that
$\oTot(W) \neq \P^{n}$.  Secondly, $\oTot(W) \neq Y$: Otherwise, $Y$
would be a $(n-2)$-dimensional variety which contains at least a
$(2n-6)$-dimensional family of lines. Thus, $Y$ would be a linear
space.  As $n \geq 4$, $X$ could not contain such a space.

We conclude that $\oTot(W) \subset \P^{n}$ is a hypersurface.  Since
$d \geq n$, \autoref{cor:sweeping} implies that $\oTot(W) \neq X$.
Observe that, by Segre's classification theorem, $\oTot(W)$ is either
a linear space, a quadric hypersurface, or a scroll.

Now pick a general point $y \in Y$.  Let $(C,y) \subset \P^{n}$ denote
the $(n-2)$-dimensional cone swept out by the family of lines
$\pi^{-1}(y)$.  Observe, by construction, that $C \subset \oTot(W)$.
Furthermore, if $\T_{y}X$ denotes the projective tangent space (an
$(n-1)$-dimensional linear space) of $X$ at $y$, then
$C \subset \T_{x}X$.

By \autoref{theorem:fundamentalform}, $\II_{y}$ defines an irreducible
quadric, and therefore for dimension reasons, $(C,y)$ must be the cone
over this quadric.  This in turn implies that $C$ spans the hyperplane
$\T_{y}X$.  Furthermore, this also eliminates the possibility that
$\oTot(W)$ is a scroll: the hypotheses of
\autoref{proposition:scrolls} would be met, so $C$ could not be a quadric.

Thus, $\oTot(W)$ is either a linear space or a quadric. Assume
$\oTot(W)$ is a linear space.  At a general point $y \in Y$, the
quadric cone $(C,y)$ is simultaneously contained in the linear space
$\oTot(W)$ and also spans $\T_{y}X$. Thus all $\T_yX$ are equal for
$y \in Y$.  This violates the finiteness of the Gauss map
\[\Gamma: X \to \P^{n*},\] which sends $x \in X$ to the hyperplane
$\T_{x}X$.

Finally, assume $\oTot(W)$ is a quadric hypersurface. Then, this
quadric hypersurface cannot be singular at the general point
$y \in Y$, because $\oTot(W)$ is irreducible. As the cone $(C,y)$ is contained in $\oTot(W)$, it follows that
\[\T_{y}\oTot(W) = \T_y X\] for $y \in Y$ general. Now, pick a general
complete curve $D \subset Y$. On the one hand, the Gauss map $\Gamma$
induces a regular map of degree $\deg(D)(d-1)$. On the other hand,
this same map agrees with the (rational) Gauss map for $\oTot(W)$,
which is a map defined by linear forms on $\P^{n}$.  Thus,
$\deg(D)(d-1) \leq \deg(D)$, which is absurd given $d \geq n$.

Thus, $Y \subset X$ cannot be a divisor, and we have exhausted all
possibilities. The theorem follows.

\end{proof}

\section{Intersection Theory computation}
\label{sec:boundinghalfspaces}
\begin{theorem}
\label{thm:2planebound}
For $d\geq 5$, the number of 2-planes in a smooth 4-fold is at most $120 d^2 - 150 d^3 + 35 d^4$. 
\end{theorem}

\begin{proof}
Let $S$ be the tautological subbundle of the Grassmannian $\mathbb{G}(1,5)$ and $Z_{5,1}\subset \mathbb{P}(S)\times_{\mathbb{G}(1,5)}\mathbb{P}(S)$ consist of the locus $(\ell,p_1,p_2)$ where $\ell$ is a line in $\mathbb{P}^5$ meeting $X$ at points $p_1$ and $p_2$ with orders 5 and 1, respectively. 

By \Cref{theorem:generalV5}, $Z_{5,1}$ is 4-dimensional. Let $H$ be the  $\mathscr{O}(1)$ class of $\mathbb{P}(S)$, which is also the hyperplane class of $\mathbb{P}^5$ pulled back under $\mathbb{P}(S)\to \mathbb{P}^5$. Let $H_1$ and $H_2$ be $H$ pulled back to the two factors of $\mathbb{P}(S)\times_{\mathbb{G}(1,5)}\mathbb{P}(S)$ respectively under the projections $\pi_1,\pi_2: \mathbb{P}(S)\times_{\mathbb{G}(1,5)}\mathbb{P}(S)\to \mathbb{P}(S)$. 

To compute $Z_{5,1}\subset \mathbb{P}(S)\times_{\mathbb{G}(1,5)}\mathbb{P}(S)$, we note that it is the zero locus of a section of the sheaf of relative principal parts $\pi_1^{*}P_{\mathbb{P}(S)/\mathbb{G}(1,5)}^4(\mathscr{O}(d))\times \pi_2^{*}\mathscr{O}(d)$ as defined in \cite[Section 11.1.1]{3264}. The relative canonical divisor of $\mathbb{P}(S)\to \mathbb{G}(1,5)$ is $-2H+\sigma_1$. Therefore, the class of $Z_{5,1}$ is
\begin{align*}
    dH_1((d-2)H_1+\sigma_1)((d-4)H_1+2\sigma_1)((d-6)H_1+3\sigma_1)((d-8)H_1+4\sigma_1)dH_2
\end{align*}
inside $A^{\bullet}(\mathbb{P}(S)\times_{\mathbb{G}(1,5)}\mathbb{P}(S))$ by \cite[Theorem 11.2]{3264}.

We will intersect $[Z_{5,1}]\in A^{\bullet}(\mathbb{P}(S)\times_{\mathbb{G}(1,5)}\mathbb{P}(S))$ with the pullback of $\mathbb{G}(1,4)\cong \sigma_{1,1}\subset \mathbb{G}(1,5)$ and the product $H_1H_2$. This must be nonnegative on each component of $Z_{5,1}$ by Kleiman's transversality \cite[Theorem 1.7]{3264} and yields
\begin{align}
    \int_{\mathbb{P}(S)\times_{\mathbb{G}(1,5)}\mathbb{P}(S)}\sigma_{1,1}H_1H_2( dH_1((d-2)H_1+\sigma_1)\cdots((d-8)H_1+4\sigma_1)dH_2=\nonumber\\
        \int_{(\mathbb{P}(S)\times_{\mathbb{G}(1,5)}\mathbb{P}(S))|_{\mathbb{G}(1,4)}}H_1H_2( dH_1((d-2)H_1+\sigma_1)\cdots((d-8)H_1+4\sigma_1)dH_2.\label{eq:2planebound}
\end{align}
In the final integral, the image of $H_1$ in $A^{\bullet}((\mathbb{P}(S)\times_{\mathbb{G}(1,5)}\mathbb{P}(S))|_{\mathbb{G}(1,4)})$ can be identified as the image of $H_1$ under the map $A^{\bullet}(\mathbb{P}^5)\to A^{\bullet}((\mathbb{P}(S)\times_{\mathbb{G}(1,5)}\mathbb{P}(S))|_{\mathbb{G}(1,4)})$ in the commutative diagram of pullback maps of Chow rings
\begin{center}
\begin{tikzcd}
A^{\bullet}(\mathbb{P}^5) \ar[r] \ar[d] & A^{\bullet}(\mathbb{P}(S)\times_{\mathbb{G}(1,5)}\mathbb{P}(S)) \ar[d]\\
A^{\bullet}(\mathbb{P}^4) \ar[r] & A^{\bullet}((\mathbb{P}(S)\times_{\mathbb{G}(1,5)}\mathbb{P}(S))|_{\mathbb{G}(1,4)}).
\end{tikzcd}
\end{center}
Since $H_1^5=0$ in $A^{\bullet}(\mathbb{P}^4)$, the same is true in $A^{\bullet}((\mathbb{P}(S)\times_{\mathbb{G}(1,5)}\mathbb{P}(S))|_{\mathbb{G}(1,4)})$. Therefore, the numerical answer in \eqref{eq:2planebound} is $O(d^4)$. To get the exact numerical value for \eqref{eq:2planebound}, we expand and evaluate with the help of Mathematica and the presentation of the Chow ring of a projective bundle \cite[Theorem 9.6]{3264} to obtain
\begin{align*}
    120 d^2 - 150 d^3 + 35 d^4.
\end{align*}
To finish, we note that every 2-plane $\Lambda$ in $X$ gives rise to a component of $Z_{5,1}$ whose class is the pullback of $\sigma_{2,2}$, the lines in $\Lambda$. Since $\sigma_{2,2}\sigma_{1,1}$ is a point on $\mathbb{G}(1,5)$, $\sigma_{2,2}\sigma_{1,1}H_1H_2$ evaluates to 1 in $A^{\bullet}(\mathbb{P}(S)\times_{\mathbb{G}(1,5)}\mathbb{P}(S))$. 
\end{proof}

\begin{remark}
If one could show that the space $Z_6$ consisting of pairs $(p,\ell)$ where $\ell$ meets $X$ to order 6 at $p$ is of expected dimension 3, then the argument in the proof of \Cref{thm:2planebound} would imply the number of 2-planes in $X$ is at most
\begin{align*}
    1800 d - 1370 d^2 + 225 d^3,
\end{align*}
which one can compare with the $15d^3$ 2-planes in $X$ in the case $X$ is a Fermat hypersurface. 

\end{remark}



\bibliographystyle{plain}
\bibliography{main}

\begin{thebibliography}{10}

\bibitem{Beheshti}
Roya Beheshti.
\newblock Lines on projective hypersurfaces.
\newblock {\em J. Reine Angew. Math.}, 592:1--21, 2006.

\bibitem{Browning2016}
T.~D. Browning, D.~R. Heath-Brown, and J.~M. Starr.
\newblock The density of rational points on non-singular hypersurfaces, {II}.
\newblock {\em Proceedings of the London Mathematical Society}, 93(2):273--303,
  2016.

\bibitem{Cheng}
Raymond Cheng.
\newblock Geometry of $q$-bic hypersurfaces.
\newblock {\em arXiv:2205.05273v1}.

\bibitem{DegItenOtt}
Alex Degtyarev, Ilia Itenberg, and John~Christian Ottem.
\newblock Planes in cubic fourfolds, arxiv:2105.13951.

\bibitem{3264}
David Eisenbud and Joe Harris.
\newblock {\em 3264 and all that---a second course in algebraic geometry}.
\newblock Cambridge University Press, Cambridge, 2016.

\bibitem{Landsberg}
J.~M. Landsberg.
\newblock Lines on projective varieties.
\newblock {\em J. Reine Angew. Math.}, 562:1--3, 2003.

\bibitem{PRT20}
Anand Patel, Eric Riedl, and Dennis Tseng.
\newblock Moduli of linear slices of high degree hypersurfaces.
\newblock {\em arXiv preprint arXiv:2005.03689}, 2020.

\bibitem{Rogora}
Enrico Rogora.
\newblock Varieties with many lines.
\newblock {\em Manuscripta Math.}, 82(2):207--226, 1994.

\bibitem{salmon}
George Salmon.
\newblock {\em A treatise on the analytic geometry of three dimensions}.
\newblock Chelsea Publishing Co., New York, 1958.
\newblock Revised by R. A. P. Rogers, 7th ed. Vol. 1, Edited by C. H. Rowe.

\bibitem{Segre}
B.~Segre.
\newblock The maximum number of lines lying on a quartic surface.
\newblock {\em Quart. J. Math. Oxford Ser.}, 14:86--96, 1943.

\bibitem{lloydSmith}
Lloyd Smith.
\newblock The kontsevich space of rational curves on cyclic covers of
  projective space.
\newblock {\em Ph.D. thesis}.

\end{thebibliography}

\end{document}